\documentclass[12pt]{amsart}
\calclayout
\usepackage[alphabetic]{amsrefs}
\usepackage{amssymb}
\usepackage{mathtools}
\usepackage{tikz-cd}
\usepackage{slashed}
\usepackage[all,cmtip]{xy}

\usepackage{hyperref}

\newcommand{\R}{\mathbb{R}}

\DeclareMathOperator{\Coeff}{Coeff}
\DeclareMathOperator{\Hom}{Hom}

\DeclareMathOperator{\ind}{ind}

\numberwithin{equation}{section}

\theoremstyle{plain}
\newtheorem{thm}[equation]{Theorem}
\newtheorem{cor}[equation]{Corollary}
\newtheorem{lem}[equation]{Lemma}

\theoremstyle{definition}
\newtheorem{defn}[equation]{Definition}

\theoremstyle{remark}
\newtheorem{rem}[equation]{Remark}

\begin{document}

\author{Honghao Gao}
\address{Honghao Gao, Yau Mathematical Sciences Center, Tsinghua University \& Beijing Institute of Mathematical Sciences and Applications} 
\email{gaohonghao@tsinghua.edu.cn}

\author{Hanming Liu}
\address{Hanming Liu, Department of Mathematics, University of Oregon}
\email{hanming@uoregon.edu}

\title{Isomorphism in the augmentation category}

\maketitle

\begin{abstract}
    Given a Legendrian submanifold in any dimension, we prove that two augmentations are isomorphic within the positive augmentation category exactly when they differ by a combination of a dga homotopy and a dilation. This extends the corresponding statement for Legendrian knots and links, but instead of relying on the dga for consistent copies, we make use of quantum flow tree techniques. Consequently, we can strengthen and clarify a result of the first author as follows: for knot contact homology, the augmentation category is not in general equivalent to the microlocal rank 1 sheaf category.
\end{abstract}

\section{Introduction}

The augmentation category is a powerful algebraic structure associated with a Legendrian submanifold in a contact manifold. It is a categorical lift of the classical theory of augmentations of the Chekanov-Eliashberg differential graded algebra (CE dga). The morphisms of the category encapsulate the bilinearized contact homology between augmentations.

The augmentation category for a Legendrian submanifold $\Lambda$ inside a contact manifold $V$ comes in two fundamental flavors: the positive augmentation category \(Aug_+(V,\Lambda)\) and the negative augmentation category \(Aug_-(V,\Lambda)\). The negative augmentation category can be derived directly from the CE dga, and it is a non-unital \(A_\infty\) category. The positive augmentation category  is a unital category, but the definition involves perturbations of multiple copies of the Legendrian. For Legendrian knots and links in \(\mathbb{R}^3\), the perturbation scheme is explicitly constructed in \cite{NRSSZ}. In higher dimensions, the definition of \(Aug_+(V,\Lambda)\) is sketched in Section 3.3 of \cite{Cha} using localization. Over a field of characteristic 2, \(\mathcal{A}ug_+(V,\Lambda)\) is independently constructed and made rigorous in the recent work \cite{Liu25}, also using localization.

Augmentations considered as objects in the unital positive augmentation category acquire a natural isomorphism in the categorical sense. The general philosophy is that the isomorphism in the augmentation is a combination of dga homotopy and dilation. The dga homotopy is the usual definition when we consider the augmentation as a dga morphism to the trivial ground dg algebra. A dilation is a tuple of invertible elements in the field, used to dilate the augmented values of Reeb chords whose endpoints are  on different link components. 

In \cite{NRSSZ}, it is proven that the categorical isomorphism between augmentations is the dga homotopy for Legendrian knots, in which case there is no dilation. In \cite{CSL+}, this general philosophy is verified for Legendrian links. Both cases are proven using the explicit formula for the $n$-copy dga. In this paper, we verify the philosophy in arbitrary dimensions. In higher dimensions, an explicit perturbative method for building a consistent $n$-copy dga is not available. Instead, we use the technique of quantum flow trees.

\medskip
Let \(\mathcal{A}\) be a semi-free dga equipped with a generating set \(\mathcal{S}\) and a \(P\)-link grading (where \(P\) is a finite set). For the convention of link grading, see \cite{Liu25}*{Definition 2.15}. Let \(k\) be a ground field. An augmentation of \(\mathcal{A}\) is a dg algebra morphism \(\epsilon:\mathcal{A}\rightarrow k\) preserving units, where \(k\) is the trivial ground dg algebra with the field \(k\) sitting at degree 0 and differentials being 0.

\begin{defn}\label{augmentation homotopy}
    Let \(\epsilon_1,\epsilon_2:\mathcal{A}\to k\) be two augmentations. A map \(K:\mathcal{A}\to k\) is said to be a \emph{dilated augmentation homotopy} if:
    \begin{enumerate}
        \item \(K\) is \(k\)-linear and has degree 1.
        \item For each \(i\in P\), there exists some \(d_i\in k^\times\) such that for each \(x\in\mathcal{S}^{ij}\), \[d_i\epsilon_1(x)-d_j\epsilon_2(x)=K(\partial x).\]
        \item \(K(xy)=K(x)\epsilon_2(y)+\epsilon_1(x)K(y)\).
    \end{enumerate}
If $d_i=1$ for every $i$, this coincides with the usual definition of dga homotopy.
\end{defn}

Suppose \(\mathcal{A}\) is the Chekanov-Eliashberg dga of a Legendrian \(\Lambda\) with a vanishing Maslov class in \(V=M\times\R\), the contactization of an exact symplectic manifold. Suppose \(\Lambda = \Lambda_1 \sqcup \dotsb \sqcup \Lambda_n\) has $n$-components, define \(\bar{\pi}(\Lambda) = \ast_{i=1}^n (\pi_1(\Lambda_i))\) to be the free product of the fundamental group of each component. The generating set \(\mathcal{S}\) consists of Reeb chords of $\Lambda$ and generators of \(\bar{\pi}(\Lambda)\). We also equip \(\mathcal{A}\) with the link grading provided by the path components of \(\Lambda\). For simplicity, we assume \(k\) has characteristic \(2\). 

\begin{thm}\label{mainthm}
    Let \(\epsilon_1,\epsilon_2:\mathcal{A}\to k\) be two augmentations. Then, \(\epsilon_1\) is quasi-isomorphic to \(\epsilon_2\) in \(\mathcal{A}ug_+(V,\Lambda)\) if and only if there is a dilated augmentation homotopy \(K\) between \(\epsilon_1\) and \(\epsilon_2\).
\end{thm}

\begin{defn}
    Suppose \(\Lambda\subset V\) is an \(n\)-component Legendrian submanifold. We say that two augmentations \(\epsilon,\epsilon_2:\mathcal{A}\to k\) are related by a \emph{dilation} if there exists an \(n\)-tuple \(\{d_1,\dotsb, d_n\} \in (k^*)^n\) such that for any mixed Reeb chord \(x_{i,j}\) from \(\Lambda_i\) to \(\Lambda_j\),
    \[
    \epsilon(x_{i,j}) = {d_j \over d_i} \epsilon_2(x_{i,j}).
    \]
\end{defn}

\begin{cor}\label{dga homotopy plus dilation}
    Let \(\epsilon,\epsilon_2:\mathcal{A}\to k\) be two augmentations. Then, \(\epsilon\) is quasi-isomorphic to \(\epsilon_2\) in \(\mathcal{A}ug_+(V,\Lambda)\) if and only if they are related by the composition of a dga homotopy and a dilation.
\end{cor}

Corollary \ref{dga homotopy plus dilation} is proved at the end of section \ref{proof of the main theorem}.

\begin{rem}
    In Corollary \ref{dga homotopy plus dilation}, it suffices to consider the composition of one dga homotopy and one dilation, or the composition of one dilation and one dga homotopy. In particular, we do not need to consider a zig-zag of them.
\end{rem}

The following statements are special cases of the above corollary.
\begin{cor}\label{specialcasecor1}
    Suppose \(\Lambda\subset V\) is a connected Legendrian submanifold. Then two augmentations \(\epsilon,\epsilon_2:\mathcal{A}\to k\) are isomorphic in the positive augmentation category if and only if they are dga homotopic.
\end{cor}

\begin{cor}\label{specialcasecor2}
    Suppose \(\Lambda\subset V\) is an \(n\)-component Legendrian submanifold whose dga \(\mathcal{A}\) has no Reeb chord of degree \(-1\). Then two augmentations \(\epsilon,\epsilon_2:\mathcal{A}\to k\) are isomorphic in the positive augmentation category if and only if they are related by a dilation.
\end{cor}

Knot contact homology \cite{Ng05a, Ng05b, Ng08, EENS} is a context in which the setup of Corollary \ref{specialcasecor2} applies. For a link with more than one component, the corollary yields that the only equivalence between augmentations is dilation. In \cite{Gao}, it is proven that the space of augmentations modulo dilations is mapped injectively, but not always surjectively, into the space of microlocal rank 1 sheaves modulo local systems. Over a field of characteristic 2, Corollary \ref{specialcasecor2} allows us to reinterpret the not necessarily surjective nature of the map defined in \cite{Gao} as:

\begin{cor}
    Over a field \(k\) of characteristic 2, for the conormal of a link in \(\R^3\), the category of microlocal rank 1 sheaves sometimes has more isomorphism classes of objects than the positive augmentation category.
\end{cor}

We do not have a precise explanation for this phenomenon at this point, but we suspect it is because the Legendrian tori arising from the link conormal bundle are not horizontally self-displaceable  in $V =J^1(S^2)$.

\section*{Acknowledgment}
We thank Lenhard Ng and Robert Lipshitz for their helpful conversations. HG is partially supported by the National Key R\&D Program of China 2023YFA1010500. HL was supported by a Johnson Fellowship from the mathematics department of the University of Oregon.

\section{Proof of the main theorem}\label{proof of the main theorem}

We fix some notation. Let \(\Lambda\) be a compact Legendrian submanifold in \(V=M\times\R\), equipped with a Maslov potential for each component of \(\Lambda\), and let \(\mathcal{A}\) be the CE dga of \(\Lambda\). We write \(\{b_i\}\) for the Reeb chords of \(\Lambda\) that have degree \(-1\) in \(\mathcal{A}\). We write \(\{e_i\}\) for the Reeb chords of \(\Lambda\) that have degree \(0\) in \(\mathcal{A}\).

Since we are only concerned with quasi-isomorphisms in the \(A_\infty\)-category \(\mathcal{A}ug_+(V,\Lambda)\), it suffices to consider the cohomology category \(H^*\mathcal{A}ug_+(V,\Lambda)\). For this, we use the characterization of \(H^*\mathcal{A}ug_+(V,\Lambda)\) given by \cite{Liu25}*{Proposition 4.11}. This characterization of \(H^*\mathcal{A}ug_+(V,\Lambda)\) can be thought of as a chain level model for \(\mathcal{A}ug_+(V,\Lambda)\), where only \(m_1\) and \(m_2\) are defined. Moreover, the \(A_\infty\)-relations that only involve \(m_1\) and \(m_2\) are satisfied. In the rest of this paper, we use this chain level model.

In \(\mathcal{A}ug_+(V,\Lambda)\), the morphism space is the cohomology of a chain complex \(\Hom(\epsilon_1,\epsilon_2)\). As a vector space, \(\Hom(\epsilon_1,\epsilon_2)\) is the free vector space with a basis given by the Reeb chords from \(\Lambda\) to a small positive pushoff \(\Lambda^+\). To specify a small positive pushoff \(\Lambda^+\), we need to pick an \(L^2\)-small positive Morse function \(f:\Lambda\to\R\). We pick \(f\) so that it has a unique critical point of index 0 on each component of \(\Lambda\). There are two kinds of Reeb chords from \(\Lambda\) to \(\Lambda^+\): short Reeb chords, which are entirely contained in a small tubular neighborhood of \(\Lambda\), and long Reeb chords, which are not contained in a small tubular neighborhood of \(\Lambda\). The short Reeb chords are in canonical bijection with the critical points of \(f\). The long Reeb chords are in canonical bijection with self Reeb chords of \(\Lambda\). For Reeb chords \(b_i\) and \(e_i\) of \(\Lambda\), we write \(b_i^\vee\) and \(e_i^\vee\) for the corresponding generators in \(\Hom(\epsilon_1,\epsilon_2)\). For a Reeb chord, its degree in \(\Hom^0(\epsilon_1,\epsilon_2)\) is \(1\) more than its degree in \(\mathcal{A}\).

\(\Hom^0(\epsilon_1,\epsilon_2)\) is the free vector space spanned by \( \{ \min_1,\ldots,\min_n,b_1^\vee,\ldots,b_m^\vee\}\), where \(\min_i\) corresponds to the short Reeb chord induced from the minimum of \(f\) on the \(i\)-th component of \(\Lambda\). The cohomology class of the sum of minimums \([\sum_{i=1}^n\min_i] \in H^0\Hom(\epsilon,\epsilon)\) is the cohomological unit, given any augmentation \(\epsilon\).

We also need to fix some auxiliary data, including a base point for each component and capping paths for the ends of Reeb chords. We choose the base points to be the same as the minima of \(f\). We choose the capping paths for the ends of an index 0 short Reeb chord to be constant, and the capping paths for the ends of an index 1 short Reeb chord to be the gradient flow of \(f\). We choose all other capping paths arbitrarily.

For any augmentations \(\epsilon_1,\epsilon_2\), an element in \(\Hom^0(\epsilon_1,\epsilon_2)\) has the following form:
\[
\sum_{i=1}^{n}\alpha_i{\min}_{i}+\sum_{j=1}^{m}K_j b_j^\vee.
\]

We would like to understand (1) when  this element is a cocycle, and (2) the composition between such elements. For these computations, we use the technique of quantum flow trees, developed in \cite{EESa}*{Section 6}, \cite{EENS}*{Section 5}, and \cite{EL}*{Section 3.3.3}. For a summary of quantum flow trees in the case of a 2-copy Legendrian, see \cite{Liu25}*{Section 2.2}.

\smallskip

(1) \textbf{Cocycle condition.}

We first compute $m_1(\sum_{i=1}^{n}\alpha_i\min_i)$. By \cite[Lemma 2.11]{Liu25}, a rigid holomorphic disk with a negative puncture at a minimum is either a Morse flow line from an index 1 critical point of \(f\) or a quantum flow tree with a constant disk.

For each index 1 critical point \(p\) of \(f\), there are 2 rigid Morse flow lines flowing toward the unique minimum $\min_i$, which lie on the same component as $p$. By our choice of basepoints and capping paths, these 2 rigid Morse flow lines contribute a total of \((1+\epsilon_1(\gamma_p)\epsilon_2(\gamma_p)^{-1})p^{\vee}\) or \((1+\epsilon_1(\gamma_p)^{-1}\epsilon_2(\gamma_p))p^{\vee}\) to $m_1(\min_i)$, where \(\gamma_p\) is the loop formed by the 2 rigid flow lines.

By \cite[Lemma 2.11]{Liu25}, for each long Reeb chord \(e_j\) that has degree 0 in the dga (which has degree 1 in \(\Hom(\epsilon_1,\epsilon_2)\)), there are 2 rigid quantum flow trees with a positive puncture at \(e_j\) and a negative puncture at a minimum \(\min_i\). One of them has a constant disk at \(e_j\) and a flow line toward \(\min_{r(e_j)}\), while the other one has a constant disk at \(e_j\) and a flow line toward \(\min_{c(e_j)}\). Recall \(r(e_j)\) and \(c(e_j)\) are the link gradings of \(e_j\) --- a Reeb chord \(e_j\) goes from component \(\Lambda_{c(e_j)}\) to component \(\Lambda_{r(e_j)}\). Assuming that \(\epsilon_1\) and \(\epsilon_2\) agree on \(\bar{\pi}(\Lambda)\), these 2 rigid quantum flow trees contribute the term \(\alpha_{c(e_j)}\epsilon_1(e_j)e_j^\vee+\alpha_{r(e_j)}e_j^\vee\epsilon_2(e_j)\) as a summand to \(m_1\left(\sum_{i=1}^{n}\alpha_i\min_i\right)\). So, assuming that \(\epsilon_1\) and \(\epsilon_2\) agree on \(\bar{\pi}(\Lambda)\),
\begin{equation}\label{m1 of alphai mini}
m_1\left(\sum_{i=1}^{n}\alpha_i{\min}_i\right)=\sum_j (\alpha_{c(e_j)}\epsilon_1(e_j)+\alpha_{r(e_j)}\epsilon_2(e_j))e_j^\vee.
\end{equation}

\smallskip
Next, we compute \(m_1(b_j^\vee)\). Since \(b_j\) is a long Reeb chord, and the negative punctures of a holomorphic disk must be shorter than the positive puncture by energy reasons, critical points of \(f\) will not appear in the differential \(m_1 (b_j^\vee)\). Therefore, the quantum flow trees that appear in this computation only contain holomorphic disks with boundary on \(\Lambda\times\R\), without flow lines.

Therefore, all the summands of \(m_1(b_j^\vee)\) can be described as follows. Let \(e_i\) be a Reeb chord \(\Lambda\) with degree \(0\) in \(\mathcal{A}\), with \(\partial e_i=\dots +xb_jy+\dots\) where \(x,y\) are monomials in \(\mathcal{A}\). Then, the summand \(xb_jy\) of \(\partial e_i\) gives rise to the summand \(\epsilon_1(x)\epsilon_2(y)e_i^\vee\) of \(m_1(b_j^\vee)\). Note that \(m_1 (b_j^\vee)\) is calculated exactly as that in bilinearized contact cohomology \cite{BC}.

Having computed $m_1$, we discuss the cocycle condition, starting with the following useful necessary condition:
\begin{lem}\label{lemma homotopy class}
If \(\alpha_i\neq 0\) for all \(1\leq i\leq n\) and \( \sum_{i=1}^{n}\alpha_i\min_i+\sum_{j=1}^{m}K_j b_j^\vee\in \Hom^0(\epsilon_1,\epsilon_2)\) is a cocycle, then \(\epsilon_1\) and \(\epsilon_2\) agree on \(\bar{\pi}(\Lambda)\).
\end{lem}
\begin{proof}

In the calculation of \(m_1 (\min_i) \), the coefficient of \(p^\vee\) vanishes if and only if \(\epsilon_1(\gamma_p)=\epsilon_2(\gamma_p)\). Since \(p^\vee\) only appears as a summand of the differential of one of the \(\min_i\) and none of the \(b_j^\vee\), we have that in 
\begin{equation}\label{Flow Coeff=0 cond}
    \text{Coeff}_{p^\vee}\left( m_1 \left(\sum_{i=1}^{n}\alpha_i{\min}_i+\sum_{j=1}^{m}K_j b_j^\vee\right)\right) =0 \;\Leftrightarrow \; \epsilon_1(\gamma_p)=\epsilon_2(\gamma_p)\text{, or }\alpha_i=0.
\end{equation}
Since the loops \(\{\gamma_p : \ind(p) =1\}\) generate \(\bar{\pi}(\Lambda)\), we conclude the assertion.
\end{proof}

We conclude the following characterization for being a cocycle.

\begin{lem}\label{lemma cocycle homotopy}
    Let \(\epsilon_1,\epsilon_2\) agree on \(\bar{\pi}(\Lambda)\). Let \(A:=\sum_{i=1}^{n}\alpha_i\min_i+\sum_{j=1}^{m}K_j b_j^\vee \in \Hom^0(\epsilon_1,\epsilon_2) \), with \(\alpha_i\neq0\) for all \(i\). View \(k\) as a dga concentrated in degree \(0\). Define a degree 1 map \(K:\mathcal{A}\to k\) on generators by \(K(b_j)=K_j,\) and on products of generators by the rule 
    \(
    K(xy)=K(x)\epsilon_2(y)+\epsilon_1(x)K(y).
    \)

    Then, \(A\) is a cocycle if and only if \(K\) is a dilated augmentation homotopy from \(\epsilon_1\) to \(\epsilon_2\).
\end{lem}

\begin{proof}

    Let \(e_l\) be a Reeb chord of \(\Lambda\) with degree \(0\) in \(\mathcal{A}\). Expanding the terms, we see that 
    \[
    \Coeff_{e_l^\vee}\left(\sum_{j=1}^{m}K(b_j^\vee) b_j^\vee\right)=K(\partial e_l).
    \]

    We have that by \ref{m1 of alphai mini}, 
    \begin{align}\label{coeff of el in m1A}
    \begin{split}
        \Coeff_{e_l^\vee}\left( m_1 A \right)
        =&\alpha_{c(e_l)}\epsilon_1(e_l)+\alpha_{r(e_l)}\epsilon_2(e_l)+\Coeff_{e_l^\vee}\left(\sum_{j=1}^{m}K(b_j^\vee) b_j^\vee\right)\\
        =&\alpha_{c(e_l)}\epsilon_1(e_l)+\alpha_{r(e_l)}\epsilon_2(e_l)+K (\partial e_l).
    \end{split}
    \end{align}
    Since \(m_1 A\) is a linear combination of \(e_l^\vee\)'s, we have that \(m_1A=0\) if and only if for every \(e_l\), \eqref{coeff of el in m1A} equals 0. It is exact (2) in Definition \ref{augmentation homotopy} of  a dilated augmentation homotopy in characteristic 2. This finishes the proof.
\end{proof}

(2) \textbf{Product $m_2$.} Consider the following two elements:
\begin{equation}\label{def:AB}
    A := \sum_{i=1}^{n}\alpha_i{\min}_i+\sum_{j=1}^{m}K_j b_j^\vee\in\Hom^0(\epsilon_1,\epsilon_2),\quad 
    B:= \sum_{i=1}^{n}\alpha'_i\mathrm{min}'_i+\sum_{j=1}^{m}K'_j {b'_j}^\vee\in\Hom^0(\epsilon_2,\epsilon_3).
\end{equation}
We compute their $m_2$ product
\begin{equation*}
    m_2\left(A,B\right)\in\Hom^0(\epsilon_1,\epsilon_3),
\end{equation*}
using quantum flow trees. By \cite[Corollary 2.12]{Liu25}, we have that 
\begin{equation}\label{m2calc1}
    m_2\left({\min}_i,{\min}_j'\right)=\delta_{ij}{\min}_i''\in\Hom^0(\epsilon_1,\epsilon_3),
\end{equation}
and that
\begin{equation}\label{m2calc2}
m_2\left(\mathrm{min}_j,{b'_i}^\vee\right)=m_2\left(b_i^\vee,\mathrm{min}'_j\right)=\delta_{ij}b''_i\in\Hom^0(\epsilon_1,\epsilon_3),
\end{equation}
where $\delta_{ij}$ is the Kronecker delta,  $\min_i''$ and ${b_i''}^\vee$ represent the corresponding generators in $\Hom^0(\epsilon_1,\epsilon_3)$.

Denote by \(E(b)\) the energy of the Reeb chord \(b\), namely the length of the chord. Assume \(b_j\) are ordered with increasing energy, namely \(E(b_1)\leq \ldots \leq E(b_m)\). Also, note that \(E(b_j)>E(\min_i)\) for any \(i,j\). Since in a holomorphic disk, the total energy of the negative punctures must be smaller than the energy of the positive puncture, we have that
\begin{equation}\label{m2calc3}
    m_2(b_i^\vee,{b'_j}^\vee)\in \mathrm{Span}_k\left\{ {b''_{\max(i,j)+1}}^\vee,\ldots,{b_m''}^\vee\right\}\subseteq\Hom^0(\epsilon_1,\epsilon_3).
\end{equation}

\begin{proof}[Proof of Theorem \ref{mainthm}]
By definition, \(\epsilon_1\) is quasi-isomorphic to \(\epsilon_2\) in \(\mathcal{A}ug_+(V,\Lambda)\), if and only if \(\epsilon_1\) is isomorphic to \(\epsilon_2\) in \(H^0\mathcal{A}ug_+(V,\Lambda)\).

(\(\Rightarrow\)) Assume that \(\epsilon_1,\epsilon_2\) are quasi-isomorphic. By definition, there exist elements $A$ and $B$ in the form of \eqref{def:AB}, which are both cocycles, and their $m_2$ product satisfy 
\begin{equation*}
m_2\left(A,B\right)=\sum_{i=1}^n\mathrm{min}''_i\in\Hom^0(\epsilon_1,\epsilon_1).
\end{equation*} 
It follows from \eqref{m2calc1} -- \eqref{m2calc3} that all \(\alpha_i,\alpha'_i\) (the coefficients of \(\min_i,\min_i'\) in \(A,B\)) are non-zero. By Lemma \ref{lemma homotopy class}, \(\epsilon_1\) agrees with \(\epsilon_2\) on \(\bar{\pi}\). We can therefore apply Lemma \ref{lemma cocycle homotopy} to conclude that there is a dilated augmentation homotopy from \(\epsilon_1\) to \(\epsilon_2\).

(\(\Leftarrow\)) Assume that there is a dilated augmentation homotopy \(K\) from \(\epsilon_1\) to \(\epsilon_2\), giving the data \(d_i\) and \(K(b_j)\). We have that \(\epsilon_1\) and \(\epsilon_2\) take the same values on \(\bar{\pi}(\Lambda)\) by definition. 

We define an element $A\in \Hom^0(\epsilon_1,\epsilon_2)$ in the form of \eqref{def:AB}, where the coefficients are defined by $\alpha_i:=d_i$ and $K_j: = K(b_j)$. Lemma \ref{lemma cocycle homotopy} implies that \(A\) is a cocycle. 

We shall construct its right inverse in $B \in \Hom^0(\epsilon_2,\epsilon_1)$ in the form \eqref{def:AB}, such that 
\[
m_2\left(A,B\right)=\sum_{i=1}^n\mathrm{min}''_i\in\Hom^0(\epsilon_1,\epsilon_1).
\]
The left inverse can be constructed analogously. It is important to emphasize that $\alpha_i$ and $K_j$ are given by the hypothesis, while $\alpha_i'$ and $K_j'$ are indeterminates. 

First, one must to define \(\alpha'_i :=\alpha_i^{-1}\). We will define \(K'_j\) by induction. By \eqref{m2calc3}, we have
\[
\mathrm{Coeff}_{{b''_1}^\vee}\left(m_2\left(A,B\right)\right)=K_1+K'_1.
\]
We define \(K'_1 :=K_1\) so as to have \(\mathrm{Coeff}_{{b''_1}^\vee}=0\). By induction, assuming we have defined \(K'_1,\ldots,K'_r\), and that for any choices of \(K'_{r+1},\ldots,K'_m\) , we have for all \(1\leq k\leq r\),
\[
\mathrm{Coeff}_{{b''_k}^\vee}\left(m_2\left(A,B\right)\right)=0.
\]
Then, we have that
\[
\mathrm{Coeff}_{{b''_{r+1}}^\vee}\left(m_2\left(A,B\right)\right)=K_{r+1}'+K_{r+1}+F_r,
\]
where \(F_r\) depends only on \(K_1,\ldots,K_r\) and \(K'_1,\ldots,K'_r\). Hence, we define \(K'_{r+1}:=K_{r+1}+F_r\). Inductively, we have defined all \(K_1',\ldots,K_m'\). 

Now we have defined $B$, it is straightforward to verify that $m_2\left(A,B\right)=\sum_{i=1}^n\mathrm{min}''_i\in\Hom^0(\epsilon_1,\epsilon_1)$. It remains to show that \(B\) is a cocycle. The \(A_\infty\) relations yield:
\[
m_1(m_2(A,B)) = m_2(m_1(A),B) + m_2(A,m_1(B)).
\]
Since \(m_2(A,B)\) is a unital element, the left hand side is 0. Since \(K\) is a dga homotopy, by Lemma \ref{lemma cocycle homotopy}, \(m_2(m_1(A),B) =0\). Hence, we have $m_2(A,m_1(B)) = 0$. Because $A$ is invertible under $m_2$, we conclude $m_1(B)= 0$.
\end{proof}

\begin{proof}[Proof of Corollary \ref{dga homotopy plus dilation}]
    A dga homotopy from \(\epsilon_1\) to \(\epsilon_2\) corresponds to an isomorphism of the form 
    \[
    \sum_{i=1}^{n}{\min}_i+\sum_{j=1}^{m}K_j b_j^\vee\in\Hom^0(\epsilon_1,\epsilon_2),
    \]
    and a dilation from \(\epsilon_2\) to \(\epsilon_3\) corresponds to an isomorphism of the form 
    \[
    \sum_{i=1}^{n}\alpha_i{\min}_i\in\Hom^0(\epsilon_2,\epsilon_3).
    \]
    Composing them results in an isomorphism, which by Theorem \ref{mainthm}, corresponds to a dilated augmentation homotopy. For reference, the dilated augmentation homotopy is
    \[    \sum_{i=1}^{n}\alpha_i{\min}_i+\sum_{j=1}^{m}\alpha_{c(b_j)}K_jb_j^\vee\in\Hom^0(\epsilon_1,\epsilon_3).
    \]
    Similarly, by composing an isomorphism corresponding to a dilated augmentation homotopy with an isomorphism corresponding to a dilation, we can obtain a dga homotopy. So, any two augmentations that are related by a dilated augmentation homotopy, are alternatively related by the composition of a dga homotopy and a dilation.
\end{proof}

\end{document}